\documentclass[12pt, reqno]{amsart}
\usepackage{amsmath, amsthm, amscd, amsfonts, amssymb, graphicx, color, xcolor}
\usepackage[bookmarksnumbered, colorlinks, plainpages]{hyperref}
\hypersetup{colorlinks=true,linkcolor=red, anchorcolor=green, citecolor=cyan, urlcolor=red, filecolor=magenta, pdftoolbar=true}

\newtheorem{theorem}{Theorem}[section]

\newtheorem{corollary}[theorem]{Corollary}
\theoremstyle{definition}
\newtheorem{definition}[theorem]{Definition}
\newtheorem{example}[theorem]{Example}

\newtheorem{remark}[theorem]{Remark}
\numberwithin{equation}{section}

\begin{document}
	
	\setcounter{page}{1}
	
	\title[The duals of $\ast$-operator frames for $End_{\mathcal{A}}^{\ast}(H)$]{The duals of $\ast$-operator frames for $End_{\mathcal{A}}^{\ast}(H)$}
	
	\author[A. Bourouihiya, M. ROSSAFI, H. LABRIGUI, A. TOURI]{A. Bourouihiya$^2$, M. ROSSAFI$^1$$^{*}$, H. LABRIGUI $^1$ \MakeLowercase{and} A. TOURI$^1$}
	
	\address{$^{1}$Department of Mathematics, University of Ibn Tofail, B.P. 133, Kenitra, Morocco}
	\email{\textcolor[rgb]{0.00,0.00,0.84}{  rossafimohamed@gmail.com;  hlabrigui75@gmail;  touri.abdo68@gmail.com}}
	
		\address{$^{2}$Department of Mathematics, Nova Southeastern University, 3301 College Avenue, Fort Lauderdale, Florida, USA}
	\email{\textcolor[rgb]{0.00,0.00,0.84}{ab1221@nova.edu}}

	\subjclass[2010]{Primary 42C15; Secondary 46L06.}
	
	\keywords{$\ast$-frame, $\ast$-operator frame, $C^{\ast}$-algebra, Hilbert $\mathcal{A}$-modules.}
	
	\date{Received: 15/09/2018; 
		\newline \indent $^{*}$Corresponding author}

\begin{abstract}
Frames play significant role in signal and image processing, which leads to many applications in differents fields. In this paper  we define the dual of $\ast$-operator frames and we show their propreties obtained in Hilbert $\mathcal{A}$-modules and we establish some results.
\end{abstract} \maketitle
Frame theory is recently an active research area in mathematics, computer science, and engineering with many exciting applications in a variety of different fields. They are generalizations of bases in Hilbert spaces. Frames for Hilbert spaces were first introduced in 1952 by Duffin and Schaeffer \cite{Duf} for study of nonharmonic Fourier series. They were reintroduced and developed in 1986 by Daubechies, Grossmann and Meyer \cite{DG}, and popularized from then on. Hilbert $C^{\ast}$-modules is a generalization of Hilbert spaces by allowing the inner product to take values in a $C^{\ast}$-algebra rather than in the ﬁeld of complex numbers. The aim of this papers is to study the dual of $\ast$-operator frames.

The paper is organized as follows:

In section 2, we briefly recall the definitions and basic properties of operator frame and $\ast$-operator frame in Hilbert $C^{\ast}$-modules.

In section 3, we introduce the  dual $\ast$-operator frame, the $\ast$-operator frame transform and the $\ast$-frame operator.

In section 4, we investigate tensor product of Hilbert $C^{\ast}$-modules, we show that tensor product of dual $\ast$-operator frames for Hilbert $C^{\ast}$-modules $\mathcal{H}$ and $\mathcal{K}$, present a dual $\ast$-operator frames for $\mathcal{H}\otimes\mathcal{K}$. 
\section{\textbf{Preliminaries}}
Let $I$ be a countable index set. In this section we briefly recall the definitions and basic properties of $C^{\ast}$-algebra, Hilbert $C^{\ast}$-modules, frame, $\ast$-frame in Hilbert $C^{\ast}$-modules. For information about frames in Hilbert spaces we refer to \cite{Ch}. Our reference for $C^{\ast}$-algebras is \cite{Dav,Con}. For a $C^{\ast}$-algebra $\mathcal{A}$, an element $a\in\mathcal{A}$ is positive ($a\geq 0$) if $a=a^{\ast}$ and $sp(a)\subset\mathbf{R^{+}}$. $\mathcal{A}^{+}$ denotes the set of positive elements of $\mathcal{A}$.

\begin{definition}\cite{Ros3} \textcolor{white}{.}
	
	A family of adjointable operators $\{T_{i}\}_{i\in I}$ on a Hilbert $\mathcal{A}$-module $\mathcal{H}$ over a unital $C^{\ast}$-algebra is said to be an operator frame for $End_{\mathcal{A}}^{\ast}(\mathcal{H})$, if there exist positive constants $A, B > 0$ such that 
	\begin{equation}
	A\langle x,x\rangle_{\mathcal{A}}\leq\sum_{i\in I}\langle T_{i}x,T_{i}x\rangle_{\mathcal{A}}\leq B\langle x,x\rangle_{\mathcal{A}}, \forall x\in\mathcal{H}.
	\end{equation}
	The numbers $A$ and $B$ are called lower and upper bound of the operator frame, respectively. If $A=B=\lambda$, the operator frame is $\lambda$-tight. If $A = B = 1$, it is called a normalized tight operator frame or a Parseval operator frame.
\end{definition}

\begin{definition}\cite{Ros3}\textcolor{white}{.}
	
	A family of adjointable operators $\{T_{i}\}_{i\in I}$ on a Hilbert $\mathcal{A}$-module $\mathcal{H}$ over a unital $C^{\ast}$-algebra is said to be an $\ast$-operator frame for $End_{\mathcal{A}}^{\ast}(\mathcal{H})$, if there exists two strictly nonzero elements $A$ and $B$ in $\mathcal{A}$ such that 
	\begin{equation}
	A\langle x,x\rangle_{\mathcal{A}} A^{\ast}\leq\sum_{i\in I}\langle T_{i}x,T_{i}x\rangle_{\mathcal{A}}\leq B\langle x,x\rangle_{\mathcal{A}} B^{\ast}, \forall x\in\mathcal{H}.
	\end{equation}
	
	The elements $A$ and $B$ are called lower and upper bounds of the $\ast$-operator frame, respectively. If $A=B=\lambda$, the $\ast$-operator frame is $\lambda$-tight. If $A = B = 1_{\mathcal{A}}$, it is called a normalized tight $\ast$-operator frame or a Parseval $\ast$-operator frame. If only upper inequality of  hold, then $\{T_{i}\}_{i\in i}$ is called an $\ast$-operator Bessel sequence for $End_{\mathcal{A}}^{\ast}(\mathcal{H})$.
	\end{definition}
If the sum in the middle of (2.1) is convergent in norm, the operator frame is called standard. If only upper inequality of (2.1) hold, then $\{T_{i}\}_{i\in I}$ is called an operator Bessel sequence for $End_{\mathcal{A}}^{\ast}(\mathcal{H})$.
\section{\textbf{Dual of $\ast$-operator Frame for $End_{\mathcal{A}}^{\ast}(\mathcal{H})$}}

We begin this section with the following definition.

\begin{definition}\textcolor{white}{.}
	
Let $\{T_{i}\}_{i\in I} \subset End_{\mathcal{A}}^{\ast}(\mathcal{H})$ be an $\ast$-operator frame for $\mathcal{H}$.
If there exists an $\ast$-operator frame  $\{\Lambda_{i}\}_{i\in I} $ such that $x=\sum_{i\in I}T^{\ast}_{i}\Lambda_{i}x$ for all $x \in \mathcal{H}$.
then the $\ast$-operator frames $\{\Lambda_{i}\}_{i\in I} $ is called the duals $\ast$-operator frames of $\{T_{i}\}_{i\in I}$.
\end{definition}
\begin{example}\textcolor{white}{.}
	
	Let $\mathcal{A}$ be a Hilbert $\mathcal{A}$-module over itself, let $ \{f_{j}\}_{j\in J}$  be an $\ast$-frame for $\mathcal{A}$.
	
	We define the adjointable $\mathcal{A}$-module map  $ \Lambda_{f_{j}} : \mathcal{A} \to \mathcal{A} $ by $\Lambda_{f_{j}}f = \langle f,f_{j} \rangle $.
	Clearly, that $\{\Lambda_{f_{j}}\}_{j\in J}$ is an $\ast$-operator frame for $\mathcal{A}$.
\end{example}

\begin{theorem}\textcolor{white}{.}
	
	Every $\ast$-operator frame for $End_{\mathcal{A}}^{\ast}(\mathcal{H})$ has a dual  $\ast$-operator frame.	
\end{theorem}
\begin{proof}\textcolor{white}{.}
	
Let $\{T_{i}\}_{i\in I} \subset End_{\mathcal{A}}^{\ast}(\mathcal{H})$ be an $\ast$-operator for $End_{\mathcal{A}}^{\ast}(\mathcal{H})$, with $\ast$-operator $S$.

We see that $\{T_{i}S^{-1}\}_{i\in I}$ is an $\ast$-operator frame.

Or, $\forall x\in \mathcal{H}$ we have :
\begin{equation*}
Sx=\sum_{i\in I}T_{i}^{\ast}T_{i}x
\end{equation*} 

then 

\begin{equation*}
x=\sum_{i\in I}T_{i}^{\ast}T_{i}S^{-1}x
\end{equation*}

hence $\{T_{i}S^{-1}\}_{i\in I}$ is a dual $\ast$-operator frame of $\{T_{i}\}_{i\in I}$.

It is called the canonique dual $\ast$-operator frame of $\{T_{i}\}_{i\in I}$.
\end{proof}
\begin{remark}\textcolor{white}{.}
	
Assume that $T=\{T_{i}\}_{i\in I}$ is an $\ast$-operator frame for $End_{\mathcal{A}}^{\ast}(\mathcal{H})$ with analytic operator $R_{T}$ and $\tilde{T}=\{\tilde{T}_{i}\}_{i\in I}$ is a dual $\ast$-operator frame of T with analytic operator $R_{T}$, then for any $x \in \mathcal{H}$ we have:

\begin{equation*}
	x=\sum_{i\in I}T_{i}^{\ast}\tilde{T}_{i}x=R^{\ast}_{T}R_{\tilde{T}}x
\end{equation*}
	
this show that every element of H can be reconstructed with a $\ast$-operator frame for $End_{\mathcal{A}}^{\ast}(\mathcal{H})$ and its dual.
\end{remark}
\begin{theorem}\textcolor{white}{.}
	
Let $\{\Lambda_{i}\}_{i\in I}$ be an $\ast$-operator frame for $End_{\mathcal{A}}^{\ast}(\mathcal{H})$ with $\ast$-operator frame transform $\theta$, the $\ast$-operator frame $S$ and the canonical dual $\ast$-operator frames $\{\tilde{\Lambda}_{i}\}_{i\in\mathbb{J}}$.

Let $\{\Omega_{i}\}_{i\in I}$ be an arbitrary dual $\ast$-operator frame of $\{\Lambda_{i}\}_{i\in I}$ with the $\ast$-operator frame transform $\eta$; then the folowing statements are true:
	\begin{itemize}
		\item [(1)] ${\theta}^{\ast}\eta=I$.
		\item [(2)] $\Omega_{i}=\Pi_{i}\eta$ for all $i \in I$.
		\item [(3)] If $\eta^{'}: \mathcal{H} \longmapsto l^{2}(\mathcal{H})$ is any adjointable right inverse of $\theta^{\ast}$ then $\{\Pi_{i}{\eta^{'}}\}_{i\in I}$ is a dual $\ast$-operator frame of  $\{\Lambda\}_{i\in I}$ with the operator frames transform $\eta^{'}$. 
		\item [(4)] The $\ast$-operator frame $S_{\Omega}$ of $\{\Omega_{i}\}_{i\in I}$ is equal to $S^{-1} + {\eta^{\ast}(I-\theta S^{-1}\theta^{\ast})\eta} $.
		\item [(5)] Every adjointable right inverse $\eta^{'}$ of $\theta^{\ast}$ is the forme :
		
		 $\eta^{'} = \theta S^{-1} + (I-\theta S^{-1} \theta^{\ast} ) \psi$ for some adjointable map $\psi : \mathcal{H} \longmapsto l^{2}(\mathcal{H}) $ and vice versa.
		 \item [(6)] There exist a $\ast$-bessel operator $\{\Delta_{j}\}_{j\in J} \in End_{\mathcal{A}}^{\ast}(\mathcal{H})$ $\{\Delta\}_{i\in\mathbb{I}}$ whose $\ast$-operator frame transform is $\eta$ and yields is $\eta$ and yields
		 
		  $\Omega_{j} = \tilde{\Lambda}_{j} + \Delta_{j} - \sum_{k\in J}\tilde{\Lambda}_{j}\Lambda^{\ast}_{k}\Delta_{k},  \forall j \in J$
				
	\end{itemize}
\end{theorem}
\begin{proof}\textcolor{white}{.}
\begin{itemize}
	\item [(1)] For $ f,g \in \mathcal{H}$ we have :
\begin{align*}
\langle \theta^{\ast}\eta f,g\rangle &= \langle \eta f,\theta g\rangle\\
&= \langle \sum_{i\in I}\Omega_{i} f,\sum_{i\in I}\Lambda_{i} g\rangle\\ 
&=\sum_{i\in I}\langle \Omega_{i} f,\Lambda_{i} g\rangle = \sum_{i\in I}\langle \Lambda^{\ast}\Omega_{i} f,g\rangle\\ 
&= \langle \sum_{i\in I}\Lambda^{\ast}\Omega_{i} f,g\rangle = \langle f, g\rangle\\ 
\end{align*}
then ${\theta}^{\ast}\eta=I$.
\item [(2)] The proof is clear from the definition
\item [(3)] Since $\eta^{'}$ is adjointable, it follows from prop3.1 that $\{\Pi_{i}{\eta^{'}}\}_{i\in I}$ is a $\ast$-bessel sequence in $\mathcal{H}$.

Also, since $(\eta^{'})^{\ast}\theta=I$; $(\eta^{'})^{\ast}$ is surjective, by lemme 2.7, for $f \in \mathcal{H}$ we have:
\begin{equation*}
||(\eta^{'})^{\ast}\eta)^{-1}||^{-1}\langle f,f\rangle \leq \langle \eta^{'} f,\eta{'} f\rangle=\sum_{i\in I}\langle \pi_{i}\eta{'}f,\pi_{i}\eta{'}f\rangle
\end{equation*}
clearly, $\eta^{'}$ is the pre-frame $\ast$-operator frame transform $\{\Pi_{i}{\eta^{'}}\}_{i\in I}$
\item [(4)]
\begin{align*}
S_{\Omega}&=\eta^{\ast}\eta\\
&=\eta^{\ast}\theta S^{-1} + \eta^{\ast}\eta - \eta^{\ast}\theta S^{-1}\\
&=\eta^{\ast}\theta S^{-1} + \eta^{\ast}\eta - \eta^{\ast}\theta S^{-1}\theta^{\ast}\eta\\
&=\eta^{\ast}\theta S^{-1} + \eta^{\ast}(I - \theta S^{-1}\theta^{\ast})\eta 
\end{align*}

\item [(5)] If $\eta^{'}$ is such a right inverse of $\theta$, then 
\begin{equation*}
\theta S^{-1} + (I - \theta S^{-1}\theta^{\ast})\eta^{'} = \theta S^{-1} + \eta{'} - \theta S^{-1}\theta^{\ast}\eta^{'} = \theta S^{-1} + \eta{'} - \theta S^{-1}I = \eta{'} 
\end{equation*}
\item [(6)] Let $ \{\Delta_{i}\}_{i\in I}$ be an $\ast$-operator bessel sequence for $End_{\mathcal{A}}^{\ast}(\mathcal{H})$ with the preframe operator $\eta$.
For  $i\in I$, let $\Omega_{i} = \tilde{\Lambda_{i}} + \Delta_{i} - \sum_{k\in I} \tilde{\Lambda_{i}}\Lambda_{k}^{\ast}\Delta_{k} $
Let $S$ and $\theta$ be the $\ast$-frame operator and the preframe operator of $ \{\Delta_{i}\}_{i\in I}$ , resp.
we define the linear operator $\psi : \mathcal{H} \longmapsto l^{2}(\mathcal{H})$ by $\psi f = (\Omega_{i}f)_{i\in I}$.
clearly, $\psi$ is adjointable, for every $i\in I$, we have 
\begin{align*}
\pi_{i}\psi &= \Omega_{i}\\
&=\Lambda_{i}S^{-1} + \Delta_{i} - \Lambda_{i}S^{-1}\sum_{k\in I} \Lambda_{k}^{\ast}\Delta_{k}\\
&= \Lambda_{i}S^{-1} + \pi_{i}\eta - \sum_{k\in I} \Lambda_{k}^{\ast}\Delta_{k}\\
&= \pi_{i}\theta S^{-1} + \pi_{i}\eta - \pi_{i}\theta S^{-1} \theta^{\ast}\eta\\
&= \pi_{i}(\theta S^{-1} + \eta - \theta S^{-1} \theta^{\ast}\eta )
\end{align*}

then 
\begin{equation*}
\psi = \theta S^{-1} + \eta - \theta S^{-1} \theta^{\ast}\eta
\end{equation*}
by parts (3) and (5) of the theorem;   $ \{\Omega_{i}\}_{i\in I}$ becomes a dual of $\ast$-operator  $ \{\Lambda_{i}\}_{i\in I}$ 
\end{itemize}
\end{proof}
\begin{example}\textcolor{white}{.}
	
	Let $\mathcal{A}$ be a Hilbert $\mathcal{A}$-module over itself, let $ \{f_{j}\}_{j\in J} \subset \mathcal{A}$.
	
	We define the adjointable $\mathcal{A}$-module map  $ \Lambda_{f_{j}} : \mathcal{A} \to \mathcal{A} $ with $\Lambda_{f_{j}}.f = \langle f,f_{j} \rangle $,
	clearly $ \{f_{j}\}_{j\in J} $ is a $\ast$-frame in $\mathcal{A}$ if and only if $ \{ \Lambda_{f_{j}}\}_{j\in J}$ is a $\ast$-operator frame in $\mathcal{A}$.

	In the folowing, we study the duals of such $\ast$-operator frame.

	(a) Let  $ \{g_{j}\}_{j\in J} \subset \mathcal{A} $ for all $f \in \mathcal{A}$ :
	
	\begin{equation*}
	\sum_{j\in J}\Lambda^{\ast}_{g_{j}}\Lambda_{f_{j}}f = \sum_{j\in J}\langle f,f_{j} \rangle g_{j}=\sum_{j\in J}\langle f,g_{j} \rangle f_{j}=\sum_{j\in J}\Lambda^{\ast}_{f_{j}}\Lambda_{g_{j}}f. 
	\end{equation*}
	Therefore, $ \{g_{j}\}_{j\in J}$ is a dual $\ast$-frame of $ \{f_{j}\}_{j\in J}$ if and only if $\{\Lambda_{g_{j}}\}_{j\in J};$ is a dual $\ast$-operator of $\{\Lambda_{f_{j}}\}_{j\in J}$
	
	(b) Let $S$ and $S_{\Lambda}$ be the $\ast$-frame operators of $ \{f_{j}\}_{j\in J}$ and  $ \{ \Lambda_{f_{j}},\mathcal{A}\}_{j\in J}$ respectively.
	
	For all $f\in \mathcal{A}$ we have:
	\begin{equation*}
	\sum_{j\in J}\langle f,f_{j} \rangle f_{j} = \sum_{j\in J}ff^{\ast}_{j}f_{j} = \sum_{j\in J}\langle \langle f,f_{j} \rangle , f^{\ast}_{j} \rangle = \sum_{j\in J}\Lambda^{\ast}_{f_{j}}\Lambda_{f_{j}}f. 
	\end{equation*}
	It follows that $S=S_{\Lambda}$
	
	(c) It is clearly to see that $ \{h_{j}\}_{j\in J} \subset \mathcal{A}$ is an $\ast$-bessel sequence  if and only if $ \{ \Lambda_{h_{j}},\mathcal{A}\}_{j\in J}$ is an $\ast$-bessel operator.
	
	(d) for a  $\ast$-bessel sequence $ \{h_{j}\}_{j\in J}$ we define 
	\begin{equation*}
	g_{j}=S^{-1}f_{j} + h_{j} - \sum_{k\in J} \langle S^{-1}f_{j} , f_{k} \rangle h_{k}
	\end{equation*}
	then the sequence  $ \{g_{j}\}_{j\in J} $ is a dual $\ast$-frame of $ \{f_{j}\}_{j\in J} $.
	
	By the last theorem, the sequence $ \{ \Gamma_{j}\}_{j\in J}$ is a dual $\ast$-operator frame of $ \{ \Lambda_{f_{j}}\}_{j\in J}$, where 
	\begin{equation*}
	\Gamma_{j} = \tilde{\Lambda}_{f_{j}} + \Lambda_{h_{j}} + \sum_{k\in J}\tilde{\Lambda}_{f_{j}}\Lambda^{\ast}_{f_{k}}\Lambda_{h_{k}},  \forall j \in J
	\end{equation*}
	now we clain that $\Gamma_{j} = \Lambda_{g_{j}}$
	
	In fact, $\forall f \in \mathcal{A}$ we have
	\begin{align*}
	\Gamma_{j}f&=\tilde{\Lambda}_{f_{j}}f + \Lambda_{h_{j}}f - \sum_{k\in J}\tilde{\Lambda}_{f_{j}}\Lambda^{\ast}_{f_{k}}\Lambda_{h_{k}}f\\
	&=\Lambda_{f_{j}}S^{-1}f + \Lambda_{h_{j}}f -  \sum_{k\in J}\Lambda_{f_{j}}S^{-1}\Lambda^{\ast}_{f_{k}} \langle f,h_{k} \rangle\\
	&=\langle S^{-1}f,f_{j} \rangle + \langle f,h_{j} \rangle -  \sum_{k\in J}\langle S^{-1}\Lambda^{\ast}_{f_{k}}\langle f,h_{k} \rangle, f_{j} \rangle\\ 
	&=\langle S^{-1}f,f_{j} \rangle + \langle f,h_{j} \rangle -  \sum_{k\in J}\langle S^{-1}\Lambda^{\ast}_{f_{k}}\Lambda_{h_{k}}f,f_{j}\rangle\\ 
	&=\langle S^{-1}f,f_{j} \rangle + \langle f,h_{j} \rangle -  \sum_{k\in J}\langle S^{-1}\Lambda_{h_{k}}ff_{k},f_{j} \rangle\\ 
	&=\langle S^{-1}f,f_{j} \rangle + \langle f,h_{j} \rangle -  \sum_{k\in J}\langle S^{-1}fh^{\ast}_{k}f_{k},f_{j} \rangle\\ 
	&=\langle f,S^{-1}f_{j} \rangle + \langle f,h_{j} \rangle -  \sum_{k\in J}\langle fh^{\ast}_{k}f_{k},S^{-1}f_{j} \rangle\\ 
	&=\langle f,S^{-1}f_{j}+h_{j} \rangle  -  \sum_{k\in J}\langle f,f^{\ast}_{k}h_{k}S^{-1}f_{j} \rangle\\ 
	&=\langle f,S^{-1}f_{j}+h_{j} -  \sum_{k\in J}\langle S^{-1}f_{j},f_{k}\rangle h_{k} \rangle\\ 
	&=\langle f,g_{j}\rangle = \Lambda_{g_{j}}f.
	\end{align*}
	
	therefore, every $\ast$-operator frame of $\{\Lambda_{f_{j}}\}_{j \in J}$ has the form :
	\begin{equation*}
	\tilde{\Lambda}_{f_{j}} + \Lambda_{h_{j}} - \sum_{k\in J}\tilde{\Lambda}_{f_{j}}\Lambda^{\ast}_{f_{k}}\Lambda_{h_{k}}
	\end{equation*}
	where $\{h_{j}\}_{j\in J}$ is a $\ast$-bessel sequence in $\mathcal{A}$.
\end{example}
\section{\textbf{Tensor product}}
In this section, we study the tensor product of the duals  $\ast$-operator frames.
\begin{theorem}\textcolor{white}{.}
	
	Let $\mathcal{H}$ and $\mathcal{K}$ are two Hilbert $C^{\ast}$-modules over unitary $C^{\ast}$-Algebras $\mathcal{A}$ and $ \mathcal{B}$ respectively,
let $ \{\Lambda_{i}\}_{i\in I} \subset End_{\mathcal{A}}^{\ast}(\mathcal{H})$  and $ \{\Gamma_{j}\}_{j\in J} \subset End_{\mathcal{B}}^{\ast}(\mathcal{K})$ are an $\ast$-operators frames.

If	$ \{\tilde{\Lambda_{i}}\}_{i\in I}$ is a dual of $ \{\Lambda_{i}\}_{i\in I}$ and 
	$ \{\tilde{\Gamma_{j}}\}_{j\in J}	$ is a dual of $ \{\Gamma_{j}\}_{j\in J}$
	
then 	$ \{\tilde{\Lambda_{i}}\otimes \tilde{\Gamma_{j}}\}_{i \in I,j \in J} $ is a dual  $\ast$-operator frame of $ \{\Lambda_{i}\otimes \Gamma_{j}\}_{i \in I,j \in J} $.
\end{theorem}
\begin{proof}\textcolor{white}{.}
	
Let $x\in \mathcal{H}$ and $ y\in \mathcal{K}$, we have :
\begin{align*}
\sum_{i \in I,j \in J} (\Lambda_{i}\otimes \Gamma_{j})^{\ast}(\tilde{\Lambda_{i}}\otimes \tilde{\Gamma_{j}})(x\otimes y)&= \sum_{i \in I,j \in J} (\Lambda_{i}^{\ast}\otimes \Gamma_{j}^{\ast})(\tilde{\Lambda_{i}}x\otimes \tilde{\Gamma_{j}}y)\\
&=\sum_{i \in I,j \in J} (\Lambda_{i}^{\ast}\tilde{\Lambda_{i}}x\otimes \Gamma_{j}^{\ast}\tilde{\Gamma_{j}}y)\\
&=\sum_{i \in I} \Lambda_{i}^{\ast}\tilde{\Lambda_{i}}x\otimes\sum_{j \in J} \Gamma_{j}^{\ast}\tilde{\Gamma_{j}}y \\
&= x\otimes y
\end{align*}
	
	then 
\begin{equation*}
	\sum_{i \in I,j \in J} (\Lambda_{i}\otimes \Gamma_{j})^{\ast}(\tilde{\Lambda_{i}}\otimes \tilde{\Gamma_{j}}) = I
\end{equation*}
	
	hence $\{\tilde{\Lambda_{i}}\otimes \tilde{\Gamma_{j}}\}_{i \in I,j \in J}$ is a dual  $\ast$-operator frames of $\{\Lambda_{i}\otimes \Gamma_{j}\}_{i \in I,j \in J}$.
		\end{proof}
	
\begin{corollary}\textcolor{white}{.}
		
		Let $(\Lambda_{ij})_{0\leq i \leq n; j\in J}$ be a family of $\ast$-operator and $ (\tilde{\Lambda}_{ij})_{0\leq i \leq n; j\in J}$ its their dual, then $(\tilde{\Lambda}_{0j}\otimes \tilde{\Lambda}_{1j}\otimes ...\otimes\tilde{\Lambda}_{nj})_{j\in J}$ is a dual of  $(\Lambda_{0j}\otimes \Lambda_{1j}\otimes ...\otimes\Lambda_{nj})_{j\in J}$.
\end{corollary}
		
\subsection*{Acknowledgment}
The authors would like to thank from the anonymous reviewers for carefully reading of the manuscript
and giving useful comments, which will help to improve the paper.


\begin{thebibliography}{99}
	\bibitem{Ch} O. Christensen, An Introduction to Frames and Riesz bases, Brikhouser,2016.
	\bibitem{Con} J.B.Conway, A Course In Operator Theory, AMS, V.21, 2000.
	\bibitem{Dav} F. R. Davidson, $\mathcal{C}^{\ast}$-algebra by example, Fields Ins. Monog. 1996.
	\bibitem{DG} I.Daubechies, A.Grossmann, and Y.Meyer, Painless nonorthogonal expansions, J.Math. Phys.27 (1986), 1271-1283
	\bibitem{Duf} R. J. Duffin, A. C. Schaeffer, A class of nonharmonic Fourier series, Trans. Amer. Math. Soc. 72 (1952),
	341-366.
	\bibitem{Ros3} M. Rossafi and S. Kabbaj, 
	\emph{$\ast$-operator frame for $End_{\mathcal{A}}^{\ast}(\mathcal{H})$}, Arxiv preprint 
	arXiv:1806.03993 (2018).
\end{thebibliography}
\end{document}